\RequirePackage{fix-cm}
\documentclass[smallextended]{svjour3}       
\smartqed  
\usepackage{graphicx}
\usepackage{amsmath, amscd, amsfonts, amssymb, graphicx, color}
\usepackage[bookmarksnumbered, colorlinks, plainpages]{hyperref}
\hypersetup{colorlinks=true,linkcolor=red, anchorcolor=green, citecolor=cyan, urlcolor=red, filecolor=magenta, pdftoolbar=true}

\begin{document}

\title{Refinements of A-numerical radius inequalities and their applications
}


\author{Pintu Bhunia,  Raj Kumar Nayak  and Kallol Paul }


\institute{Pintu Bhunia \at
               Department of Mathematics, Jadavpur University\\
							Kolkata 700032, West Bengal, India\\
              \email{pintubhunia5206@gmail.com}
           \and
           Raj Kumar Nayak \at
              Department of Mathematics, Jadavpur University\\
							Kolkata 700032, West Bengal, India\\
							\email{rajkumarju51@gmail.com}
						\and	
						Kallol Paul \at
							Department of Mathematics, Jadavpur University\\
							Kolkata 700032, West Bengal, India\\
							\email{kalloldada@gmail.com}
							}

\date{Received: date / Accepted: date}

\maketitle

\begin{abstract}
We present sharp lower bounds for the A-numerical radius of semi-Hilbertian space operators. We also present an upper bound. Further we compute new upper bounds for the $B$-numerical radius of $2 \times 2$ operator matrices where $B = \textit{diag}(A,A)$, $A$ being a positive operator. As an application of the A-numerical radius inequalities, we obtain a bound for the zeros of a polynomial which is quite a bit improvement of some famous existing bounds for the zeros of  polynomials.
\end{abstract}

\section{Introduction}
\label{intro}
\noindent Let $\mathcal{H}$ be a non trivial complex Hilbert space with inner product $\langle .,. \rangle$ and $\|.\|$ be the norm induced from $\langle .,. \rangle$.
Let $\mathcal{B}(\mathcal{H})$ denote the $C^*$-algebra of all bounded linear operators acting on $\mathcal{H}.$  In this article, by an operator we mean a bounded linear operator. Let $T\in \mathcal{B}(\mathcal{H})$. The numerical range  of $T$ is the set of all scalars given by: 
\begin{eqnarray*}
W(T) &=& \{\langle Tx,x\rangle: x\in \mathcal{H}, \|x\|=1\}.
\end{eqnarray*}
It is well-known that $W(T)$ is convex and the closure of $W(T)$ contains the spectrum  of $T$ which is denoted by $\sigma (T)$. These are two of the most important properties of the numerical range (see \cite {GR}). The numerical radius of $T$ is defined as:
\begin{eqnarray*}
w(T) &=& \sup_{\|x\|=1} \left |\langle Tx,x\rangle \right|.
\end{eqnarray*}
The spectral radius of $T$ is given by: 
\[r(T)=\sup\{|\lambda|: \lambda \in \sigma (T)\}.\]
It is clear to see that $r(T) \leq w(T).$ 
In this article, we study the generalization of this numerical radius, i.e., known as the A-numerical radius of semi-Hilbertian space operators. For this the following definitions, notations and terminologies are essential.

\begin{definition} 
An operator $A\in \mathcal{B}(\mathcal{H})$ is called positive if $\langle Ax,x \rangle \geq 0$ for all $x\in \mathcal{H}$ and  is called strictly positive if $\langle Ax,x \rangle > 0$ for all non-zero $x\in \mathcal{H}$.
\end{definition}
\noindent  For a positive (strictly positive) operator $A$ we write $A\geq 0$ $(A>0)$. We used $A$ to signify a positive operator defined  on $\mathcal{H}$. Then $A$ induces a positive semi-definite sesquilinear form $\langle .,. \rangle _A : \mathcal{H} \times \mathcal{H} \rightarrow \mathbb{C}$ defined as $\langle x,y \rangle _A=\langle Ax,y \rangle$ for all $x,y\in \mathcal{H}$. Let $\|.\|_A$ denote the semi-norm on $\mathcal{H}$ induced from the  sesquilinear form $\langle .,. \rangle_A,$ i.e., $\|x\|_A=\sqrt{\langle x,x \rangle_A}$ for all $x\in \mathcal{H}.$ It is easy to verify that $\|.\|_A$ is a norm if and only if $A>0$. Given $T\in \mathcal{B}(\mathcal{H})$, if there exists $c>0$ such that $\|Tx\|_A\leq c\|x\|_A$ for all $x\in  \overline{\mathcal{R}(A)}$ where $\mathcal{R}(A)$ is the range of $A$, then  the A-operator semi-norm of $T$ is given by: 
\[\|T\|_A=\sup_{x\in  \overline{\mathcal{R}(A)},x\neq 0}\frac{\|Tx\|_A}{\|x\|_A}< +\infty.\] 
The symbols $I$ and $O$ stands for the identity operator and the zero operator defined on $\mathcal{H}$, respectively. 
The generalization of the numerical range, known as the A-numerical range (see \cite{BFA}), is denoted by $W_A(T)$ and defined as: 
\[W_A(T) = \{\langle Tx,x\rangle_A: x\in \mathcal{H}, \|x\|_A=1\}.\] 
 
\noindent The A-numerical radius $w_A(T)$ and the A-Crawford number $m_A(T)$ of $T$ are given by: 
\[ w_A(T)=\sup\{|\lambda|: \lambda \in W_A(T)\}, \]
\[m_A(T)=\inf\{|\lambda|: \lambda \in W_A(T)\}.\]

\noindent In general, $w_A(T)$ can be equal to $+\infty$ for an arbitrary operator $T\in \mathcal{B}(\mathcal{H}).$ Indeed, one can take the operator $A=\left(\begin{array}{cc}
0&0\\
0&1
\end{array}\right)$ and $T=\left(\begin{array}{cc}
0&1\\
1&0
\end{array}\right)$ on $\mathbb{C}^2.$ In addition, $\|T\|_A$ can be equal to $+\infty$ for an arbitrary operator $T\in \mathcal{B}(\mathcal{H}).$\\

\noindent For a given $T\in \mathcal{B}(\mathcal{H})$, if there exists $c>0$ such that $\|Tx\|_A \leq c\|x\|_A$ for all $x\in  \mathcal{H}$, then the A-numerical radius $w_A(T)$ of $T$ satisfies the following inequality, (see \cite{F,Z}):
\[\frac{1}{2}\|T\|_A \leq w_A(T)\leq \|T\|_A.\]
\noindent Various A-numerical radius inequalities improving this above inequality have been studied in \cite{BP}, \cite{BPN}, \cite{BFP}, \cite{KF}, \cite{MXZ} and \cite{Z}. Now we recall the following definition:
\begin{definition}
For $T \in \mathcal{B}(\mathcal{H})$, an operator $R \in \mathcal{B}(\mathcal{H})$ is called an $A$-adjoint of $T$ if for every $x,y \in \mathcal{H}$ such that $\langle Tx,y \rangle_A=\langle x,Ry \rangle_A$, i.e., $AR=T^*A$ where $T^*$ is the adjoint of $T$.
\end{definition}

\noindent Neither the existence nor the uniqueness of $A$-adjoint of $T$ holds true in general. The set of all operators in  $\mathcal{B}^{A}(\mathcal{H})$ which admit $A$-adjoints is denoted by $\mathcal{B}_{A}(\mathcal{H})$, where  $\mathcal{B}^{A}(\mathcal{H})=\{T \in \mathcal{B}(\mathcal{H}) : \|T\|_A < +\infty\}$. By Douglas Theorem \cite{doug}, we  have
\begin{align*}\label{badeh}
\mathcal{B}_{A}(\mathcal{H})
& = \left\{T\in \mathcal{B}(\mathcal{H})\,:\;\mathcal{R}(T^{*}A)\subseteq \mathcal{R}(A)\right\}\nonumber\\
 &=\left\{T \in \mathcal{B}(\mathcal{H})\,:\;\exists \,\lambda > 0\,\textit{such that}\;\|ATx\| \leq \lambda \|Ax\|,\;\forall\,x\in \mathcal{H}  \right\}.
\end{align*}
If $T\in \mathcal{B}_A(\mathcal{H})$, the reduced solution of the equation
$AX=T^*A$ is a distinguished $A$-adjoint operator of $T$, which is denoted by $T^{\sharp_A}$. Note that, $T^{\sharp_A}=A^\dag T^*A$ in which $A^\dag$ is the Moore-Penrose inverse of $A$.


\begin{definition}
An operator $T\in \mathcal{B}(\mathcal{H})$ is said to be A-self-adjoint if $AT$ is self-adjoint, i.e., $AT=T^*A$ and it is called A-positive if $AT\geq 0.$
\end{definition}

\begin{definition}
An operator $U\in  \mathcal{B}_A(\mathcal{H})$ is said to be $A$-unitary if $\|Ux\|_A=\|x\|_A$ and $\|U^{\sharp_A}x\|_A=\|x\|_A$ for all $x\in \mathcal{H}$.
\end{definition}

\noindent For $T\in \mathcal{B}_A(\mathcal{H})$, if $U\in \mathcal{B}_A(\mathcal{H})$ be an $A$-unitary then the following equality for the A-numerical range (see \cite{BFA}) holds: $$W_A(T)=W_A(UTU^{\sharp_A}).$$

\noindent We know that if $T\in \mathcal{B}_A(\mathcal{H})$ then $T^{\sharp_A}\in \mathcal{B}_A(\mathcal{H})$ and $(T^{\sharp_A})^{\sharp_A}=PTP$, where $P$ is the  orthogonal projection onto $\overline{\mathcal{R}(A)}$. For $T\in \mathcal{B}_A(\mathcal{H})$, it is useful to recall that the operators  $T^{\sharp_A}T$ and $TT^{\sharp_A}$ both are A-positive operators satisfying 
$$\|T^{\sharp_A}T\|_A =\|TT^{\sharp_A}\|_A =\|T\|^2_A =\|T^{\sharp_A}\|^2_A.$$  For more information on this we refer \cite{AS} and \cite{ACG2}. We note the following  properties which are used repeatedly in this article. For $T,S \in  \mathcal{B}_A(\mathcal{H})$,  $(TS)^{\sharp_A}=S^{\sharp_A}T^{\sharp_A}$, $\|TS\|_A\leq \|T\|_A\|S\|_A$ and  $\|Tx\|_A\leq \|T\|_A\|x\|_A$, for all $x\in \mathcal{H}$. For $T\in \mathcal{B}_A(\mathcal{H})$, we write $\textit{Re}_A(T)=\frac{1}{2}(T+T^{\sharp_A})$ and $\textit{Im}_A(T)=\frac{1}{2i}(T-T^{\sharp_A})$. In \cite[Lemma 2.3]{Z}, Zamani showed that if $T\in \mathcal{B}_A(\mathcal{H})$ then  for all $\theta \in \mathbb{R},$
$$w_A(\textit{Re}_A(e^{i\theta}T))=\|\textit{Re}_A(e^{i\theta}T)\|_A.$$
In \cite[Th. 2.5]{Z}, Zamani also proved that if $T\in \mathcal{B}_A(\mathcal{H})$ then 
$$ w_A(T)=\sup_{\theta \in \mathbb{R}}\left\|\textit{Re}_A(e^{i\theta}T) \right\|_A.$$


\noindent Using the above A-numerical radius equality of semi-Hilbertian space operators, we obtain new lower bounds for the A-numerical radius of semi-Hilbertian space operators which improve on the bounds in \cite{Z}, recently obtained by Zamani. We find a new upper bound of the A-numerical radius for semi-Hilbertian space operators. Also we obtain some new bounds for the B-numerical radius of $2 \times 2$ operator matrices where $B=\textit{diag} (A,A)$. In the last section, we obtain a bound for zeros of a monic polynomial using the A-numerical radius of semi-Hilbertian space operators.

\section{\textbf{A-numerical radius bounds}}
\label{sec:1}

\noindent In this section we obtain some new bounds for the generalized numerical radius, i.e., the A-numerical radius of semi-Hilbertian space operators which are refinement of existing bounds.  We begin this section with the following lower bound.
\begin{theorem}\label{th-2.1}
Let $T\in \mathcal{B}_A(\mathcal{H})$. Then $$w_A(T)\geq \sqrt{\|\textit{Re}_A(T)\|_A^2+ m_A^2(\textit{Im}_A(T))}.$$ 
\end{theorem}
\begin{proof} 
For all $x\in \mathcal{H}$, we have that $\langle \textit{Re}_A(T)x,x\rangle_A, \langle \textit{Im}_A(T)x,x\rangle_A \in \mathbb{R}$ because $$\langle \textit{Re}_A(T)x,x\rangle_A=\frac{1}{2}\langle Tx,x\rangle_A+\frac{1}{2}\overline {\langle Tx,x\rangle_A},$$ $$\langle \textit{Im}_A(T)x,x\rangle_A=\frac{1}{2i}\langle Tx,x\rangle_A-\frac{1}{2i}\overline {\langle Tx,x\rangle_A}.$$ Therefore, there exists a sequence $\{x_n\}$ in $\mathcal{H}$ with $\|x_n\|_A=1$ such that $$\langle \textit{Re}_A(T)x_n,x_n\rangle_A \rightarrow \lambda \in \mathbb{R} \,\,~~\textit{and}~~\,\, |\lambda |=w_A(\textit{Re}_A(T)).$$ Also $\|\textit{Re}_A(T)\|_A=w_A(\textit{Re}_A(T))$, so  $\|\textit{Re}_A(T)\|_A=|\lambda|$.  Now, we have 
\begin{eqnarray*}
|\langle Tx_n,x_n \rangle_A |^2 &=&|\langle (\textit{Re}_A(T)+i\textit{Im}_A(T))x_n,x_n \rangle_A|^2 \\
&=&|\langle \textit{Re}_A(T)x_n,x_n \rangle_A +i \langle \textit{Im}_A(T)x_n,x_n \rangle_A|^2 \\
 &=& \langle \textit{Re}_A(T)x_n,x_n \rangle_A^2 + \langle \textit{Im}_A(T)x_n,x_n \rangle_A^2\\
&\geq& |\langle \textit{Re}_A(T)x_n,x_n \rangle_A|^2 +m_A^2(\textit{Im}_A(T)).
\end{eqnarray*}
Since $w_A(T)=\sup_{\|x\|_A=1}|\langle Tx,x\rangle_A|,$ so we get

\begin{eqnarray*}
w^2_A(T) &\geq& |\langle Tx_n,x_n \rangle_A |^2\\
&\geq& |\langle \textit{Re}_A(T)x_n,x_n \rangle_A|^2 +m_A^2(\textit{Im}_A(T)).
\end{eqnarray*}
Taking supremum over $\|x_n\|_A=1$, $x_n\in \mathcal{H}$,  we get

\begin{eqnarray*}
w^2_A(T)  &\geq& \lambda^2 +m_A^2(\textit{Im}_A(T))\\
\Rightarrow {w^2_A(T)}&\geq&  {\|\textit{Re}_A(T)\|_A^2 +m_A^2(\textit{Im}_A(T))}.
\end{eqnarray*}
This is the required inequality of the theorem.
\end{proof}
Using an argument similar to that in the proof of Theorem \ref{th-2.1}, we can prove the following theorem.
\begin{theorem}\label{th-2.2}
Let $T\in \mathcal{B}_A(\mathcal{H})$. Then $$w_A(T)\geq \sqrt{\|\textit{Im}_A(T)\|_A^2+ m_A^2(\textit{Re}_A(T))}.$$
\end{theorem}

\begin{remark}\label{re-2.3}
Zamani in \cite[Cor. 2.7]{Z} proved that $w_A(T)\geq \|\textit{Re}_A(T)\|_A$ and $w_A(T)\geq \|\textit{Im}_A(T)\|_A$. Clearly the inequalities obtained by us in Theorem \ref{th-2.1} and Theorem \ref{th-2.2} are better than the  corresponding inequalities $w_A(T)\geq \|\textit{Re}_A(T)\|_A$ and $w_A(T)\geq \|\textit{Im}_A(T)\|_A$. Consider $T= \left(\begin{array}{cc}
1+i & 0 \\
0 & 2+i 
\end{array}\right)$ and  $A= \left(\begin{array}{cc}
1 & 0 \\
0 & 1 
\end{array}\right)$. Then Theorem \ref{th-2.1} gives $w_A(T)\geq \sqrt{5}$, whereas $w_A(T)\geq \|\textit{Re}_A(T)\|_A$ gives  $w_A(T)\geq 2.$ Also Theorem \ref{th-2.2} gives $w_A(T)\geq \sqrt{2}$, whereas $w_A(T)\geq \|\textit{Im}_A(T)\|_A$ gives  $w_A(T)\geq 1.$
\end{remark}

Next we give an upper bound for the A-numerical radius of semi-Hilbertian space operators. 

\begin{theorem}\label{theorem:upperbound2}
Let $T \in \mathcal{B}_A(\mathcal{H})$. For any $\phi \in [0,2\pi)$, let $ H^A_{\phi}=\textit{Re}_A(e^{i \phi} T)$.  Then \[w^2_A(T) \leq {{\|H^A_\phi\|_A}^2+{\|H^A_{\phi+\frac{\pi}{2}}\|_A}^2}.\]
\end{theorem}

\begin{proof}
 Clearly $H^A_{\theta}=\cos\theta \textit{Re}_A(T)-\sin\theta \textit{Im}_A(T)$, so for an any $\phi \in [0,2\pi)$, we get 
\begin{eqnarray*}
 H^A_{\theta+\phi} &=& \cos(\theta+\phi)\textit{Re}_A(T)-\sin(\theta+\phi)\textit{Im}_A(T)\\
   &=& \cos\theta [\cos\phi \textit{Re}_A(T)-\sin\phi \textit{Im}_A(T)]-\sin\theta [-\cos(\phi+\frac{\pi}{2}) \textit{Re}_A(T) \\
	 & &  +  \sin(\phi+\frac{\pi}{2}) \textit{Im}_A(T)]\\
	&=&\cos\theta \textit{Re}_A(e^{i\phi}T)+\sin\theta \textit{Re}_A(e^{i(\phi+\frac{\pi}{2})}T)\\
	 &=& H^A_\phi \cos\theta + H^A_{\phi+\frac{\pi}{2}} \sin\theta \\
	\Rightarrow \|H^A_{\theta+\phi}\|_A &\leq &\|H^A_\phi\cos\theta \|_A+\|H^A_{\phi+\frac{\pi}{2}}\sin\theta \|_A\\
	 \Rightarrow \|H^A_{\theta+\phi}\|_A  &\leq & \sqrt{{\| H^A_\phi\|_A}^2+{\| H^A_{\phi+\frac{\pi}{2}}\|_A}^2}.
\end{eqnarray*}
Taking supremum over $\theta \in \mathbb{R}$, we get
\[ w_A^2(T) \leq  {{\| H^A_\phi\|_A}^2+{\| H^A_{\phi+\frac{\pi}{2}}\|_A}^2}.\] This completes the proof.
\end{proof}

\begin{remark}\label{remark:remark1} We see that Theorem \ref{theorem:upperbound2} holds for all $\phi \in [0,2\pi)$, so we get
\[w^2_A(T) \leq \inf_{\phi \in [0,2\pi)} \{{\|H^A_\phi\|_A}^2+{\|H^A_{\phi+\frac{\pi}{2}}\|_A}^2\}.\]
Noting that for $ \phi = 0, \| H^A_\phi \|_A = \| \textit{Re}_A(T) \|_A $ and $\| H^A_{\phi + \pi/2}\|_A = \| \textit{Im}_A(T) \|_A,$  it follows from the above inequality that  $$w_A^2(T) \leq {\|\textit{Re}_A(T)\|_A^2+\|\textit{Im}_A(T)\|_A^2}.$$
\end{remark}

Next we compute an upper bound for B-numerical radius  of $2 \times 2$ operator matrices where $B=\textit{diag} (A,A)$.  Here we note that the operator $\left(\begin{array}{cc}
T_{11} & T_{12} \\
T_{21} & T_{22} 
\end{array}\right)$ is in $\mathcal{B}_B(\mathcal{H} \oplus \mathcal{H})$ if the operator $T_{ij}$  (for $i,j=1,2$) are in $\mathcal{B}_A(\mathcal{H})$ and in the case (see \cite[Lemma 3.1]{BFP})
$$ \left(\begin{array}{cc}
T_{11} & T_{12} \\
T_{21} & T_{22} 
\end{array}\right)^{\sharp_B}=\left(\begin{array}{cc}
T^{\sharp_A}_{11} & T^{\sharp_A}_{21} \\
T^{\sharp_A}_{12} & T^{\sharp_A}_{22} 
\end{array}\right).$$ 

Next we prove the following inequality.

\begin{lemma}\label{lemma-2.4}
Let $T_{11}, T_{12} \in \mathcal{B}_A(\mathcal{H})$ and $B=\left(\begin{array}{cc}
A&O\\
O&A
\end{array}\right)$. Then
\[w_B\left(\begin{array}{cc}
T_{11}&T_{12} \\
O& O
\end{array}\right) \leq  \frac{1}{2}\Big[w_A(T_{11})+\sqrt{w^2_A(T_{11})+\|T_{12}\|_A^2} \Big].\]
\end{lemma}

\begin{proof}
Let $T=\left(\begin{array}{cc}
T_{11}&T_{12} \\
O& O
\end{array}\right).$
Then for any $\theta \in \mathbb{R}$, and using \cite[Lemma 4.15]{BP} we have
\begin{eqnarray*}
\| \mbox{Re}_B(e^{i\theta}T)\|_B&=& \left\|\left(\begin{array}{cc}
\mbox{Re}_A(e^{i\theta}T_{11})&\frac{1}{2}e^{i\theta}T_{12} \\
\frac{1}{2}(e^{i\theta}T_{12})^{{\sharp}_A}& O
\end{array}\right)\right\|_B\\
&\leq& \left\|\left(\begin{array}{cc}
\|\mbox{Re}_A(e^{i\theta}T_{11})\|_A&\frac{1}{2}\|e^{i\theta}T_{12}\|_A \\
\frac{1}{2}\|(e^{i\theta}T_{12})^{{\sharp}_A}\|_A& 0
\end{array}\right)\right\|\\
&\leq& \left \|\left(\begin{array}{cc}
w_A(T_{11})&\frac{1}{2}\|T_{12}\|_A \\
\frac{1}{2}\|T_{12}\|_A & 0
\end{array}\right)\right\|\\
&=& \frac{1}{2}\left[w_A(T_{11})+\sqrt{w_A^2(T_{11})+\|T_{12}\|_A^2}\right].
\end{eqnarray*} 
This completes the proof.
\end{proof}

By using an argument similar to that in the proof of Lemma \ref{lemma-2.4}, we can prove the following inequality. To prove this inequality, first we need the following lemma which can be found in \cite[Lemma 3.8]{BFP}.

\begin{lemma}\label{Bhunia et. al.}
Let $T\in \mathcal{B}_A(\mathcal{H})$ and $U$ be an $A$-unitary operator on $\mathcal{H}$. Then $$w_A(U^{\sharp_A}TU)=w_A(T).$$
\end{lemma}

\begin{theorem}\label{th-2.5}
Let $T_{11}, T_{12}, T_{21},T_{22} \in \mathcal{B}_A(\mathcal{H})$ and $T=\left(\begin{array}{cc}
T_{11}&T_{12} \\
T_{21}& T_{22}
\end{array}\right)$.  If $B=\left(\begin{array}{cc}
A&O\\
O&A
\end{array}\right)$ then

\begin{eqnarray*}
w_B\left(T\right) &\leq&  \frac{1}{2}\left [w_A(T_{11})+w_A(T_{22})+\sqrt{w^2_A(T_{11})+\|T_{12}\|_A^2}+\sqrt{w^2_A(T_{22})+\|T_{21}\|_A^2} \right].
\end{eqnarray*}
\end{theorem}

\begin{proof}
Let $U=\left(\begin{array}{cc}
O&I \\
I& O
\end{array}\right).$ Then it is clear to see that for all $x \in \mathcal{H} \oplus \mathcal{H}$, $\|Ux\|_B=\|x\|_B$ and $\|U^{\sharp_B}x\|_B=\|x\|_B$. So, $U$ is a $B$-unitary operator on $\mathcal{H} \oplus \mathcal{H}$. Therefore by using Lemma \ref{Bhunia et. al.} we have 
\begin{eqnarray*}
w_B(T) &\leq & w_B\left(\begin{array}{cc}
T_{11}&T_{12} \\
O& O
\end{array}\right)+w_B\left(\begin{array}{cc}
O&O \\
T_{21}& T_{22}
\end{array}\right)\\
&=& w_B\left(\begin{array}{cc}
T_{11}&T_{12} \\
O& O
\end{array}\right)+w_B\left(U^{\sharp_B}\left(\begin{array}{cc}
O&O \\
T_{21}& T_{22}
\end{array}\right)U\right)\\
&=& w_B\left(\begin{array}{cc}
T_{11}&T_{12} \\
O& O
\end{array}\right)+w_B\left(\begin{array}{cc}
T_{22}&T_{21}\\
O& O
\end{array}\right).
\end{eqnarray*}
Using Lemma \ref{lemma-2.4} in the above inequality, we get the required inequality of the theorem. 
\end{proof}

\begin{remark}\label{re-2.6}
Recently, Feki \cite[Cor. 2.2]{F} proved that if $T \in \mathcal{B}_A(\mathcal{H})$ with $AT^2=O$ then $w_A(T)=\frac{1}{2}\|T\|_A.$ So $ w_B\left(\begin{array}{cc}
O&T_{12} \\
O& O
\end{array}\right)= \frac{1}{2}\left \|\left(\begin{array}{cc}
O&T_{12} \\
O& O
\end{array}\right) \right \|_B$. Since $T_{12} \in \mathcal{B}_A(\mathcal{H})$, so there exists a sequence of A-unit vectors $\{y_n\}$ in $\mathcal{H}$, i.e., $\|y_n\|_A=1$ such that $\|T_{12}y_n\|_A \rightarrow \|T_{12}\|_A.$ By a simple calculation we have $\left \|\left(\begin{array}{cc}
O&T_{12} \\
O& O
\end{array}\right) \left(\begin{array}{c}
0 \\
y_n
\end{array}\right)\right \|_B=\|T_{12}y_n\|_A \rightarrow \|T_{12}\|_A.$ This implies that $ \left \|\left(\begin{array}{cc}
O&T_{12} \\
O& O
\end{array}\right) \right \|_B $ $\geq \|T_{12}\|_A.$ Also by a simple calculation we have for any B-unit vector $(x_1,x_2) \in \mathcal{H} \oplus \mathcal{H}$, 
$\left \|\left(\begin{array}{cc}
O&T_{12} \\
O& O
\end{array}\right) \left(\begin{array}{c}
x_1\\
x_2
\end{array}\right)\right \|_B = \|T_{12}x_2\|_A\leq \|T_{12}\|_A\|x_2\|_A \leq \|T_{12}\|_A.$ This implies that $ \left \|\left(\begin{array}{cc}
O&T_{12} \\
O& O
\end{array}\right) \right \|_B \leq \|T_{12}\|_A.$ So, $ \left \|\left(\begin{array}{cc}
O&T_{12} \\
O& O
\end{array}\right) \right \|_B = \|T_{12}\|_A.$
Here we would like to remark that if we consider $T_{11}=T_{21}=T_{22}=O$ then the inequalities in Lemma \ref{lemma-2.4}  and Theorem \ref{th-2.5} admits of an equality, i.e., $ w_B\left(\begin{array}{cc}
O&T_{12} \\
O& O
\end{array}\right)=\frac{\|T_{12}\|_A}{2}.$ 
\end{remark}

\begin{remark}\label{Remark-bp}
We give an example to show the bound obtained in Theorem \ref{th-2.5} is better than the upper bound obtained in \cite[Th. 3.4]{BP}. If we consider $A=I$, $T_{11}=T_{22}=O$, $T_{12}=I$ and $T_{21}=-2I$ then Theorem \ref{th-2.5} gives $w_B\left(\begin{array}{cc}
T_{11}&T_{12} \\
T_{21}& T_{22}
\end{array}\right)\leq \frac{3}{2}$, whereas \cite[Th. 3.4]{BP} gives $w_B\left(\begin{array}{cc}
T_{11}&T_{12} \\
T_{21}& T_{22}
\end{array}\right)\leq 2.$
\end{remark}

Next we obtain another upper bound.

\begin{theorem}\label{th-2.7}
Let $T_{11}, T_{12}, T_{21},T_{22} \in \mathcal{B}_A(\mathcal{H})$ and $B=\left(\begin{array}{cc}
A&O\\
O&A
\end{array}\right)$. Then
\begin{eqnarray*}
w_B\left(\begin{array}{cc}
T_{11}&T_{12} \\
T_{21}& T_{22}
\end{array}\right) &\leq&  \frac{1}{2}w_A(T_{11})+w_A(T_{22})\\
&& +\frac{1}{2}\sqrt{t^2w^2_A(T_{11})+\|T_{12}\|_A^2} +\frac{1}{2}\sqrt{(1-t)^2w^2_A(T_{11})+\|T_{21}\|_A^2},
\end{eqnarray*}
where $0\leq t \leq 1.$
\end{theorem}

\begin{proof}
Let $T=\left(\begin{array}{cc}
T_{11}&T_{12} \\
T_{21}& T_{22}
\end{array}\right)$. Using \cite[Th. 4.12]{BP} we have for any $\theta \in \mathbb{R}$,
\begin{eqnarray*}
2 \|\textit{Re}_B(e^{i\theta}T)\|_B &=& 2 w_B(\textit{Re}_B(e^{i\theta}T))\\
&=& w_B\left(\begin{array}{cc}
2\textit{Re}_A(e^{i\theta}T_{11})&e^{i \theta}T_{12}+(e^{i \theta}T_{21})^{\sharp_A} \\
e^{i \theta}T_{21}+(e^{i \theta}T_{12})^{\sharp_A}& 2\textit{Re}_A(e^{i\theta}T_{22})
\end{array}\right)\\
&\leq& w_B\left(\begin{array}{cc}
2t\textit{Re}_A(e^{i\theta}T_{11})&e^{i \theta}T_{12} \\
(e^{i \theta}T_{12})^{\sharp_A}& O
\end{array}\right)\\
&& +w_B\left(\begin{array}{cc}
2(1-t)\textit{Re}_A(e^{i\theta}T_{11})&(e^{i \theta}T_{21})^{\sharp_A} \\
e^{i \theta}T_{21}& O
\end{array}\right)\\
&& + w_B\left(\begin{array}{cc}
O&O \\
O&2\textit{Re}_A(e^{i\theta}T_{22})
\end{array}\right)\\
&\leq& w\left(\begin{array}{cc}
2tw_A(\textit{Re}_A(e^{i\theta}T_{11}))&\|e^{i \theta}T_{12}\|_A \\
\|(e^{i \theta}T_{12})^{\sharp_A}\|_A& 0
\end{array}\right)\\
&& +w\left(\begin{array}{cc}
2(1-t)w_A(\textit{Re}_A(e^{i\theta}T_{11}))&\|(e^{i \theta}T_{21})^{\sharp_A}\|_A \\
\|e^{i \theta}T_{21}\|_A& 0
\end{array}\right)\\
&& + w\left(\begin{array}{cc}
0&0 \\
0&2w_A(\textit{Re}_A(e^{i\theta}T_{22}))
\end{array}\right)\\
&\leq & w\left(\begin{array}{cc}
2t w_A(T_{11})&\|T_{12}\|_A \\
\|T_{12}\|_A& 0
\end{array}\right)\\
&& + w\left(\begin{array}{cc}
2(1-t) w_A(T_{11})&\|T_{21}\|_A \\
\|T_{21}\|_A& 0
\end{array}\right)\\
&&+ 2w_A(T_{22})\\
\end{eqnarray*}

\begin{eqnarray*}
&=& t w_A(T_{11})+\sqrt{t^2w_A^2(T_{11})+\|T_{12}\|_A^2} \\
&& +(1-t) w_A(T_{11})+\sqrt{(1-t)^2w_A^2(T_{11})+\|T_{21}\|_A^2} \\
&& + 2w_A(T_{22})\\
&\leq&  w_A(T_{11})+2w_A(T_{22})\\
&& +\sqrt{t^2w^2_A(T_{11})+\|T_{12}\|_A^2} +\sqrt{(1-t)^2w^2_A(T_{11})+\|T_{21}\|_A^2}.
\end{eqnarray*}
Taking supremum over $\theta \in \mathbb{R}$, we get
\begin{eqnarray*}
w_B(T) &\leq&  \frac{1}{2}w_A(T_{11})+w_A(T_{22})+\frac{1}{2}\sqrt{t^2w^2_A(T_{11})+\|T_{12}\|_A^2} \\
&& \hspace{4cm}+\frac{1}{2}\sqrt{(1-t)^2w^2_A(T_{11})+\|T_{21}\|_A^2}.
\end{eqnarray*}
This completes the proof.
\end{proof}

Using an argument similar to that in the proof of Theorem \ref{th-2.7}, we can prove the following inequality.

\begin{theorem}\label{th-2.8}
Let $T_{11}, T_{12}, T_{21},T_{22} \in \mathcal{B}_A(\mathcal{H})$ and $B=\left(\begin{array}{cc}
A&O\\
O&A
\end{array}\right)$. Then
\begin{eqnarray*}
w_B\left(\begin{array}{cc}
T_{11}&T_{12} \\
T_{21}& T_{22}
\end{array}\right) &\leq&  \frac{1}{2}w_A(T_{22})+w_A(T_{11})\\
&& +\frac{1}{2}\sqrt{t^2w^2_A(T_{22})+\|T_{21}\|_A^2} +\frac{1}{2}\sqrt{(1-t)^2w^2_A(T_{22})+\|T_{12}\|_A^2},
\end{eqnarray*}
where $0\leq t \leq 1.$
\end{theorem}

\begin{remark}\label{re-2.8}
(i) If we consider $A=I$, then we can show with examples that the inequality in Theorem \ref{th-2.7} ( also Theorem \ref{th-2.8})  improve on the inequalities \cite[Cor. 3.1]{S} and \cite[Cor. 3.4]{S}, obtained by Shebrawi. Consider $T_{11}=T_{12}=T_{21}=I, T_{22}=O$, then Theorem \ref{th-2.7} (for $t=\frac{1}{2}$) gives  $w\left(\begin{array}{cc}
T_{11}&T_{12} \\
T_{21}& T_{22}
\end{array}\right)\leq \frac{1+\sqrt{5}}{2}$,  but \cite[Cor. 3.1]{S} and \cite[Cor. 3.4]{S} both gives $w\left(\begin{array}{cc}
T_{11}&T_{12} \\
T_{21}& T_{22}
\end{array}\right)\leq \frac{2+\sqrt{2}}{2}$. Also if we consider $T_{22}=T_{12}=T_{21}=I, T_{11}=O$, then Theorem \ref{th-2.8} (for $t=\frac{1}{2}$) gives better bound than the bounds in\cite[Cor. 3.1]{S} and \cite[Cor. 3.4]{S}.\\
(ii) If we consider the same example of Remark \ref{Remark-bp} then we see that the bounds obtained in Theorem \ref{th-2.7} and Theorem \ref{th-2.8} improve on the right hand inequality in \cite[Th. 3.4]{BP}.
\end{remark}

\section{\textbf{Application: Estimation of zeros of a polynomial}}
\label{sec:2}

\noindent One can easily compute the exact roots of an $n$ degree polynomial for $n \leq 4$ but for degrees that are higher than $4$, we can merely ascertain a bound to the roots. Humbling as this endeavour might seem, over the years, various researchers have dedicated their attention with both vigour and rigour to this end. Following suit, we have come up with our own determination of bounds using $A$-numerical radius of an operator acting  on semi-Hilbertian space $(\mathcal{H}, \langle.,.\rangle_A)$,  where $A>0$, which in turn act as a refinement of those proposed by some of the former mathematicians in this field. Let $p(z)=\sum_{j=0}^na_jz^j$ be a polynomial, where $a_j\in \mathbb{C}$ and $a_n=1.$ Let $C(p)$ is Frobenius companion matrix (see \cite{BBP}) of $p(z)$ where  
\begin{eqnarray*}
  C(p)&=&\left(\begin{array}{ccccccc}
    -a_{n-1}&-a_{n-2}&\ldots&-a_1&-a_0 \\
    1&0&\ldots&0&0\\
		0&1&\ldots&0&0\\
    \vdots\\
    0&0&\ldots&1&0
    \end{array}\right).
		\end{eqnarray*}
It is well-known that all roots of $p(z)$ are eigenvalues of $C(p)$. Before the proof of our main result in this section, we first give some famous bounds for zeros of $p(z)$. Let $\mu $ be a zero of $p(z).$\\
(1) Cauchy \cite{CM} gives
		   \[ |\mu|\leq 1+\max \left \{ |a_i|:i=0,1 \ldots,n-1 \right\}=R_C.\]
(2) Carmichael and Mason \cite{CM} gives 
		\[ |\mu|\leq \left (1+\sum_{i=0}^{n-1}|a_i|^2 \right)^{\frac{1}{2}}=R_{CM}.\]
(3) Fujii and Kubo \cite{FK} gives
			\[|\mu|\leq \frac{1}{2}\big[\big(\sum_{j=0}^{n-1}|a_j|^2\big)^{\frac{1}{2}}+|a_{n-1}|\big]+\cos\frac{\pi}{n+1}=R_{FK}.\]
			
\noindent Now we prove our main result in this section. To prove this, the following proposition is essential.

\begin{proposition}\label{pr-3.1}
Let $T\in \mathcal{B}_A(\mathcal{H})$ and there exists $m>0$ such that $A \geq mI >0.$ Then $\sigma(T)\subseteq \overline{W_A(T)}.$
\end{proposition}
\begin{proof}
Since the boundary of the spectrum $\sigma (T)$ of $T$ is contained in the approximate point spectrum $\sigma_{app}(T)$ of $T$ (see \cite[p. 6]{GR}) and $W_A(T)$ is convex subset of $\mathbb{C}$ (see \cite[Th. 2.1]{BFA}), so it is sufficient to prove that $\sigma_{app}(T)\subseteq \overline{W_A(T)}.$ Let $\lambda \in \sigma_{app}(T)$. Then there exists a sequence of unit vectors $\{x_n\}$ in $\mathcal{H}$ such that $\|(T-\lambda I)x_n\|\rightarrow 0.$ Since $A\geq mI$, so  $\|x_n\|_A^2= \langle Ax_n,x_n \rangle \geq m.$ So, $\frac{1}{\|x_n\|_A}\leq \frac{1}{\sqrt{m}}$ for all $n$. Therefore, there exists a subsequence $\{c_{n_k}\}=\{\frac{1}{\|x_{n_k}\|_A}\}$ such that $c_{n_k}\rightarrow c$ for some real scalar $c.$ Using this fact we have $\left \|(T-\lambda I)\frac{x_{n_k}}{\|x_{n_k}\|_A} \right \|\rightarrow 0.$ So we get
\begin{eqnarray*}
\left \|(T-\lambda I)\frac{x_{n_k}}{\|x_{n_k}\|_A} \right \|_A &=& \sqrt{\left \langle A(T-\lambda I)\frac{x_{n_k}}{\|x_{n_k}\|_A},(T-\lambda I)\frac{x_{n_k}}{\|x_{n_k}\|_A} \right \rangle } \rightarrow  0.\\
\end{eqnarray*}	
	
By Cauchy Schwarz inequality, we get	
	
\begin{eqnarray*}	
\left | \left \langle (T-\lambda I)\frac{x_{n_k}}{\|x_{n_k}\|_A},\frac{x_{n_k}}{\|x_{n_k}\|_A} \right \rangle _A \right |
&\leq &	\left \| (T-\lambda I)\frac{x_{n_k}}{\|x_{n_k}\|_A} \right \|_A \rightarrow  0\\
\Rightarrow  \left \langle T \frac{x_{n_k}}{\|x_{n_k}\|_A},\frac{x_{n_k}}{\|x_{n_k}\|_A} \right \rangle _A &\rightarrow & \lambda.\\
\end{eqnarray*}	
 So, $\lambda \in \overline{W_A(T)}.$ Therefore $\sigma_{app}(T)\subseteq \overline{W_A(T)}.$ 
\end{proof}

\begin{remark}\label{pr-3.2}
Here we note that if we consider $A\geq 0$ in Proposition \ref{pr-3.1}, then $\sigma(T)$ may not be contained in $\overline{W_A(T)}.$ Because there may exist an $x_0(\neq 0) \in \mathcal{H}$ such that $Tx_0=\lambda x_0$ and $Ax_0=0$ and then $\|x_0\|_A=0.$ So, $\langle Tx_0,x_0 \rangle_A \notin \overline{W_A(T)}. $\\
As for example, we consider $T=\left(\begin{array}{cc}
1&0\\
0&3
\end{array}\right)$ and $A=\left(\begin{array}{cc}
1&0\\
0&0
\end{array}\right ).$ Then $W_A(T)=\{1\}$ and $\sigma(T)=\{1,3\}$. So in this case,  $\sigma(T)$ is not contained in $\overline{W_A(T)}.$\\
\end{remark}

\noindent We now prove our desired bound for zeros of $p(z)$.		
\begin{theorem}\label{th-3.3}
Let $\mu$ be a zero of $p(z).$  Then $$|\mu| \leq \max \big\{\alpha_1, \alpha_2, \alpha_3, \ldots, \alpha_n \big\}=R_{PRK},$$
where 
\begin{eqnarray*}
\alpha_1&=&\frac{1}{d_1}\left[\frac{d_1}{2}(|a_{n-1}|+\sum_{i=0}^{n-1}|a_i|)+\frac{d_2}{2}\right],\\
\alpha_2&=&\frac{1}{d_2}\left[\frac{d_1}{2}|a_{n-2}|+\frac{d_2}{2}+\frac{d_3}{2}\right],\\
\alpha_3&=&\frac{1}{d_3}\left[\frac{d_1}{2}|a_{n-3}|+\frac{d_3}{2}+\frac{d_4}{2}\right],\\
\vdots &&\\
\alpha_{n-1}&=&\frac{1}{d_{n-1}}\left[\frac{d_1}{2}|a_{1}|+\frac{d_{n-1}}{2}+\frac{d_n}{2}\right],\\
\alpha_{n}&=&\frac{1}{d_n}\left[\frac{d_1}{2}|a_{0}|+\frac{d_n}{2}\right], 
\end{eqnarray*}
and $d_1,d_2,\ldots,d_n$ are arbitrary real strictly positive constants.
\end{theorem}

\begin{proof}
Let $C=C(p)$ and $A=\textit{diag}~(d_1,d_2,\ldots,d_n).$ Let $x=(x_1,x_2,\ldots,x_n)^t\in \mathbb{C}^n$ with $\|x\|_A=1$, i.e., $\sum_{i=1}^nd_i|x_i|^2=1.$\\
Now, $\langle Cx,x\rangle_A= \langle ACx,x \rangle=-d_1(a_{n-1}|x_1|^2+a_{n-2}x_2\overline{x_1}+\ldots+a_1x_{n-1}\overline{x_1}+a_0x_{n}\overline{x_1})+d_2x_1\overline{x_2}+d_3x_2\overline{x_3}+\ldots+d_nx_{n-1}\overline{x_n}.$\\  Therefore,

\begin{eqnarray*} 
\mid \langle Cx,x \rangle_A  \mid &\leq& \frac{d_1}{2}\big[2|a_{n-1}||x_1|^2+|a_{n-2}|(|x_2|^2+|x_1|^2)+\ldots+|a_0|(|x_{n}|^2+|x_1|^2)\big] \\
&&+ \frac{d_2}{2}(|x_2|^2+|x_1|^2)\\
&&+ \frac{d_3}{2}(|x_3|^2+|x_2|^2)\\
&& \vdots\\
&&+ \frac{d_{n-1}}{2}(|x_{n-1}|^2+|x_{n-2}|^2)\\
&&+ \frac{d_n}{2}(|x_{n}|^2+|x_{n-1}|^2)\\
&=& d_1\alpha_1|x_1|^2+d_2\alpha_2|x_2|^2+\ldots+d_n\alpha_n|x_n|^2\\
&\leq & \max \big\{\alpha_1, \alpha_2, \alpha_3, \ldots, \alpha_n \big\}.
\end{eqnarray*}
Taking supremum over $\|x\|_A=1, x\in \mathbb{C}^n$, we get
$$w_A(C) \leq \max \big\{\alpha_1, \alpha_2, \alpha_3, \ldots, \alpha_n \big\}.$$ 
Therefore, from Proposition \ref{pr-3.1} we get, $|\mu|\leq w_A(C)$ and this gives our required bound for zeros of $p(z)$.
\end{proof}

\begin{remark}\label{re-3.4} 
We finally give an example to show that the bound for zeros obtained here is better than some of existing bounds. Consider  $p(z)=z^5+3z^2+\frac{1}{100}z+\frac{1}{10},$  then we have 
$R_C=4,~~ R_{CM}\approx 3.1638, ~~R_{FK}\approx 2.3668.$ But Theorem \ref{th-3.3} gives $|\mu| \leq R_{PRK}\approx 2.0833,$ by taking $d_1=2, d_2=1, d_3=2, d_4=\frac{1}{3}, d_5=1.$ 		
\end{remark}

\begin{acknowledgements}
The authors would like to thank the referees for their insightful suggestions that helped us to improve this article. Pintu Bhunia and Raj Kumar Nayak would like to thank UGC, Govt. of India for the financial support in the form of senior research fellowship. Prof. Kallol Paul would like to thank RUSA 2.0, Jadavpur University for the partial support.
\end{acknowledgements}

%
%



\end{document}